\documentclass[12pt,a4paper]{article}
\usepackage{amsfonts}
\usepackage{amsmath}
\usepackage{amssymb}
\usepackage{fullpage}
\usepackage{amsthm}
\usepackage{graphicx}
\usepackage{tikz}
\usepackage{tikz-cd}

\newtheorem{theorem}{Theorem}[section]
\newtheorem{proposition}[theorem]{Proposition}
\newtheorem{lemma}[theorem]{Lemma} 
\newtheorem{corollary}[theorem]{Corollary}
\newtheorem{conjecture}[theorem]{Conjecture}
\newtheorem{remark}[theorem]{Remark}

\newtheorem{definition}[theorem]{Definition}

\newcommand{\bmu}{{\underline{\mu}}}

\newcommand{\nn}{{\mathbb N}}

\newcommand{\ff}{{\mathbb F}}
\newcommand{\zz}{{\mathbb Z}}
\newcommand{\rr}{{\mathbb R}}

\newcommand{\sss}{{\mathbb S}}
\newcommand{\ttt}{{\mathbb T}}
\newcommand{\bg}{{\mathbf{g}}}

\newcommand{\calr}{{\mathcal{R}}}

\newcommand{\semi}{{\rtimes}}
\newcommand{\link}{{\mathrm{Lk}}}
\newcommand{\Star}{{\mathrm{St}}}

\title{On the virtual and residual properties of a generalization
  of Bestvina-Brady groups} 
\author{Ian J. Leary
\and
Vladimir Vankov}

\date{\today}

\begin{document}

\maketitle

\begin{abstract}

  Previously one of us introduced a family of groups $G^M_L(S)$,
  parametrized by a finite flag complex $L$, a regular covering $M$ of
  $L$, and a set $S$ of integers.  We give conjectural descriptions of
  when $G^M_L(S)$ is either residually finite or virtually
  torsion-free.  In the case that $M$ is a finite cover and $S$ is
  periodic, there is an extension with kernel $G_L^M(S)$ and infinite
  cyclic quotient that is a CAT(0) cubical group.  We conjecture that
  this group is virtually special.  We relate these three 
  conjectures to each other and prove many cases of them.  
\end{abstract}

\section{Introduction} 

Bestvina-Brady groups are a family of infinite discrete groups
that were constructed in the 1990s to answer a long standing open
question in homological group theory~\cite{bb}; the existence
of non-finitely presented groups of type $FP$.  In~\cite{ufp},
one of us generalized the Bestvina-Brady construction, producing 
an uncountable family of groups of type $FP$.  Further results
concerning these groups can be found in~\cite{kls,bks}.  Our aim
is to study some of the other, non-homological, properties of these
groups.

It is well-known that Bestvina-Brady groups are torsion-free,
residually finite and linear over~$\zz$.  We address the question
of when the groups introduced in~\cite{ufp} have these and other
related properties.  In order to state our results, we first need
to say a little about the construction of the groups.  

Bestvina-Brady groups are parametrized by a finite flag simplicial
complex, and we denote by $BB_L$ the group corresponding to the
complex $L$.  The map $L\mapsto BB_L$ can be viewed as a functor from
the category of non-empty flag complexes and simplicial maps to
the category of groups.

The groups $G_L^M(S)$ introduced in~\cite{ufp} are parametrized by a
finite connected flag simplicial complex $L$, together with a
connected regular (possibly infinite) covering $M\rightarrow L$ of $L$
and a set $S\subseteq \zz$.  For the applications to homological group
theory the main case of interest is when $M$ is the universal covering,
so~\cite{ufp} focussed mainly on that case, but see the discussion in
\cite[section~21]{ufp} for the general case.  The group of deck
transformations of the regular covering $M\rightarrow L$ plays a
major role in describing $G_L^M(S)$, so we introduce the notation
$\pi(M,L)$ for this group.  Of course, a choice of basepoints
identifies $\pi(M,L)$ with the factor group 
$\pi_1(L)/\pi_1(M)$ of the two fundamental groups.

For fixed $L$ and $M$, the groups $G_L^M(S)$ interpolate between two
groups that are easily described in terms of Bestvina-Brady groups:
$G_L^M(\zz)$ is $BB_L$ and $G_L^M(\emptyset)$ is the semidirect
product $BB_M\semi \pi(M,L)$, where the action of $\pi(M,L)$ on $M$ is
used to define its conjugation action on $BB_M$.  (Usually
Bestvina-Brady groups are defined only for finite complexes, because
this is the case in which the homological finiteness properties of
$BB_L$ are controlled by the homology of~$L$, but the definition makes
sense for arbitrary flag complexes such as~$M$.)

The case $S=\zz$ is in some ways an exception, as will
become apparent in the statement of some of our results.  This is
because for every other $S$, $G_L^M(S)$ contains subgroups isomorphic
to $\pi(M,L)$.  Another exceptional case is when $M$ is the trivial
covering of $L$, or equivalently $\pi(M,L)=\{1\}$; in this case
$G_L^L(S)=BB_L$ is independent of $S$.

It is not hard to decide which of the
groups $G_L^M(S)$ are torsion-free.

\begin{proposition}\label{prop:torfree}
  The group $G_L^M(S)$ is torsion-free if and only if either
  $S=\zz$ or $\pi(M,L)$ is torsion-free.
\end{proposition}

We give necessary conditions for $G_L^M(S)$ to be virtually
torsion-free and to be residually finite.  In the statement,
a subset $S$ of $\zz$ is said to be periodic if there exists
$n>0$ so that $S+n=S$, and the least such $n$ is called the
period of the set $S$.

\begin{theorem}\label{thm:properties}
  If $G_L^M(S)$ is virtually torsion-free then at least
  one of the following holds: 
      \begin{itemize}
      \item{}
        $S=\zz$; 
      \item{}
        $\pi(M,L)$ is torsion-free;
      \item{}
        $\pi(M,L)$ is virtually torsion-free and $S$ is periodic.
      \end{itemize}
      If $G_L^M(S)$ is residually finite then at least one of
      the following holds:
      \begin{itemize}
      \item{}
        $S=\zz$;
      \item{}
        $\pi(M,L)=\{1\}$; 
      \item{}
        $\pi(M,L)$ is residually finite and $S$ is closed
        in the profinite topology on $\zz$.    
      \end{itemize}
\end{theorem}

It seems plausible that these necessary conditions may also be
sufficient, and so we make the following conjectures.  

\begin{conjecture}\label{conj:vtorfree}
  If $S$ is periodic and $\pi(M,L)$ is virtually torsion-free then
  $G_L^M(S)$ is virtually torsion-free.
\end{conjecture}

\begin{conjecture}\label{conj:resfin}   
  If $S$ is closed in the profinite topology on $\zz$ and $\pi(M,L)$ is
  residually finite then $G_L^M(S)$ is residually finite.
\end{conjecture}

Some cases of the first part of Theorem~\ref{thm:properties} appeared
as~\cite[thm.~3.1]{vv}, and~conjecture~1.3 of~\cite{kv} discusses another
context in which groups that are parametrized by subsets of $\zz$ are expected
to be virtually torsion-free if and only if the subset is periodic;
interestingly the opposite implication is the one that remains open for 
those groups.  In the 1970's Dyson defined a family of groups $L(S)$ for
$S\subseteq \zz$ as amalgamations of two copies of the lamplighter group,
and she showed that $L(S)$ is residually finite if and only if $S$ is
closed in $\zz$~\cite{dyson}.  The connection between residual finiteness
of a group parametrized by $S\subseteq \zz$ and the set $S$ being closed
arises in~\cite{dyson} for much the same reason as in our work.  

We offer some evidence for Conjectures
\ref{conj:vtorfree}~and~\ref{conj:resfin}.  Firstly, we offer
a reduction to a smaller family of cases.  

\begin{theorem}\label{thm:reduction}
  If Conjecture~\ref{conj:vtorfree} or Conjecture~\ref{conj:resfin}
  holds whenever $\pi(M,L)$ is finite and $S$ is periodic, then
  it holds in all cases.
\end{theorem}

Secondly, we establish all the conjectures under some hypotheses
on the covering.  

\begin{theorem}\label{thm:vtorfree}
  Let $\Gamma$ be a simplicial graph, obtained by subdividing each
  edge of another graph into at least $r$ pieces, and let
  $\Delta\rightarrow \Gamma$ be any finite regular covering of $\Gamma$.
  Suppose that $M$ is a connected component of $M_0$, defined as
  the pullback
  \[\begin{tikzcd}
    M_0\arrow[r] \arrow[d]
    &\Delta\arrow[d] \\
    L\arrow[r]&\Gamma
  \end{tikzcd}\]
  for some simplicial map $L\rightarrow \Gamma$.  
  For $r\geq 4$ Conjecture~\ref{conj:vtorfree} holds for $(M,L)$
  and for $r\geq 12$ Conjecture~\ref{conj:resfin} holds for $(M,L)$.
\end{theorem}

The hypotheses on the covering $M\rightarrow L$ split naturally into
two parts, one topological and one combinatorial.  The topological
hypothesis is that there is a graph $\overline{\Gamma}$ and a finite regular
covering $p:\overline{\Delta}\rightarrow \overline{\Gamma}$, together with a map
$f:|L|\rightarrow |\overline{\Gamma}|$ of topological realizations, such
that $|M|$ is a connected component of the pullback covering.  We recall that
the pullback is the regular covering of $|L|$ defined as
$\{(x,y)\in |L|\times |\overline{\Delta}|\,\,:\,\,f(x)=p(y)\}$, with
the covering map $(x,y)\mapsto x$.  The
combinatorial hypothesis is that the triangulation of $L$
(and hence also of $M$) is sufficiently fine that $f$ is homotopic
to a simplicial map from $L$ to a suitable subdivision $\Gamma$ of
$\overline{\Gamma}$.

In the following corollary, we replace the topological hypothesis by
a hypothesis that involves only the fundamental groups $\pi_1(L)$
and $\pi_1(M)$, at the expense of making the combinatorial hypothesis
far less explicit.  

\begin{corollary}\label{cor:vtorfree}
  Suppose that there is a homomorphism $f:\pi_1(L)\rightarrow F$, for
  $F$ a free group and a finite-index normal subgroup $N\triangleleft
  F$ so that $\pi_1(M)=f^{-1}(N)$.  Then both conjectures
  \ref{conj:vtorfree}~and\ref{conj:resfin} hold for sufficiently fine
  subdivisions $(M',L')$ of the pair $(M,L)$.
\end{corollary}

The distinction between the topological and combinatorial hypotheses
is a useful one.  If $(M',L')$ is a subdivision of $(M,L)$, the
homological finiteness properties of $G_{L'}^{M'}(S)$ are similar to those
of $G_L^M(S)$.  So from the point of view of constructing examples,
the combinatorial hypotheses that we make can be ignored.  However, 
we warn the reader that there may be a topological
obstruction to each of our conjectures, although we have been unable
to construct any counterexamples.  In particular, in the case
when $L$ is a flag triangulation of the projective plane $\rr P^2$
and $M$ its universal cover, we have been able to establish
Conjecture~\ref{conj:vtorfree} only for a small number of choices
of $S$, including $S=2\zz$ which is a special case of
Proposition~\ref{prop:homology}.  We see $L=\rr P^2$ as an
important test case for Conjecture~\ref{conj:vtorfree}.  

The proofs of many of our results use an action of the group
$G_L^M(S)$ on a CAT(0) cubical complex $X_L^M(S)$, which generalizes
the action of $BB_L$ on the universal cover of the Salvetti complex
for the right-angled Artin group $A_L$.  The group $G_L^M(S)$ acts
freely except that some vertex stabilizers are isomorphic to
$\pi(M,L)$.  In particular, the action is proper if and only if
$\pi(M,L)$ is finite.  The action of $G_L^M(S)$ on $X_L^M(S)$
has infinitely many orbits of vertices and so is never cocompact.
However, in the case when $S$ is periodic of period $n$ there is
a larger cocompact group of cubical automorphisms of $X_L^M(S)$ which
we denote by $G_L^M(S)\semi n\zz$.  The action of $G_L^M(S)\semi n\zz$
on $X_L^M(S)$ has $n$ orbits of vertices and the group contains
$G_L^M(S)$ as a normal subgroup with infinite cyclic quotient.  

In the case when both $\pi(M,L)$ is finite and $S$ is periodic of
period $n>0$, the group $G_L^M(S)\semi n\zz$ is
CAT(0) cubical in the sense that it acts properly and cocompactly on
the CAT(0) cube complex $X_L^M(S)$.  We make a third conjecture
concerning this case.

\begin{conjecture}\label{conj:special}
  If $\pi(M,L)$ is finite and $S+n=S$ for some $n>0$ then the CAT(0)
  cubical group $G_L^M(S)\semi n\zz$ is virtually special in the sense
  of~\cite{hw}.  
\end{conjecture}

Haglund and Wise showed that virtually special groups are virtually
torsion-free and residually finite~\cite{hw}, and hence 
Conjecture~\ref{conj:special} implies our other conjectures.  

\begin{proposition}
  Conjecture~\ref{conj:special} for the complex $L$ implies
  Conjectures \ref{conj:vtorfree}~and~\ref{conj:resfin} for
  the complex $L$.
\end{proposition}

Rather more surprisingly, we show that, at least from the topological
viewpoint, there is no obstruction to Conjecture~\ref{conj:special}
provided that Conjecture~\ref{conj:vtorfree} holds.  

\begin{theorem}\label{thm:vtorfreeimpliesvspec}
  Suppose that $\pi(M,L)$ is finite and that $S+n=S$ for some $n>0$.
  If $G_L^M(S)$ is virtually torsion-free, then for any sufficiently
  fine subdivision $(M',L')$ of the pair $(M,L)$, the group
  $G_{L'}^{M'}(S)\semi n\zz$ is virtually special.  In particular,
  Conjecture~\ref{conj:vtorfree} for $L$ implies Conjectures
  \ref{conj:special}~and~\ref{conj:resfin} for $L'$.  The second
  barycentric subdivision is sufficiently fine.  
\end{theorem}

To show that $G_L^M(S)$ is virtually torsion-free, we rely on group
presentations and carefully chosen maps to finite groups.  However,
most of our other results rely heavily on studying the cube complex
$X_L^M(S)$.  For example, to prove that a non-identity element $g\in
G_L^M(S)$ has non-identity image in $G_L^N(T)$ for some finite cover
$N$ and periodic $T\supseteq S$ we show that we can choose $N$ and $T$
so that the geodesic in $X_L^M(S)$ from a base vertex $x_0$ to $gx_0$
projects to a geodesic in $X_L^N(T)$.  All our results concerning
residual finiteness rely on proving cases of
Conjecture~\ref{conj:special}.  Like the action of $BB_L$ on the
universal covering of the Salvetti complex, the action of $G_L^M(S)$
on the CAT(0) cube complex $X_L^M(S)$ is never cocompact, but it
does have only finitely many orbits of hyperplanes.  To show that
$X_L^M(S)/H$ is non-cocompact special 
for some torsion-free finite-index normal subgroup $H\leq G_L^M(S)$ we
use the action of $Q=G_L^M(S)/H$ on the complex.  Edges of
$X_L^M(S)/H$ are in free $Q$-orbits, but some of the vertices are in
non-free orbits.  As an example, to show that a hyperplane in
$X_L^M(S)/H$ cannot directly self-osculate we consider the stabilizer
in $Q$ of the hyperplane.  Provided that this stabilizer has trivial
intersection with each vertex stabilizer no direct inter-osculation
can occur.

After reviewing some background material, we prove
Theorem~\ref{thm:properties} in Section~\ref{sec:three} and
Theorem~\ref{thm:reduction} in Section~\ref{sec:four}.  Sections
\ref{sec:five}~and~\ref{sec:six} complete the proof of
Theorem~\ref{thm:vtorfreeimpliesvspec} and Section~\ref{sec:seven}
completes the proof of Theorem~\ref{thm:vtorfree}.  Sections
\ref{sec:eight}~and~\ref{sec:nine} describe further examples,
and in Section~\ref{sec:ten} we use our results to construct
groups with surprising combinations of properties.

Much of this work was done while the second named author was working on
his PhD under the supervision of the first named author; further
related work appears in the second named author's PhD
thesis~\cite{vvtwo} and in~\cite{vv}.  The authors thank the
referee for their comments on an earlier version of this article.

\section{Background}\label{sec:back}

A flag complex is a simplicial complex $L$ with the property that
every finite clique within its edge graph spans a simplex.  The
barycentric subdivision of any simplicial complex is flag.  The
right-angled Artin group associated to a flag complex is the
group with generators the vertices of $L$, subject only to the
relations that the ends of each edge commute:
\[A_L=\langle v\in L^0 \,\,:\,\, [v,w]=1 \,\,(v,w)\in L^1\rangle.\]
This construction is functorial in $L$, in the sense that a
simplicial map $f:M\rightarrow L$ induces a group homomorphism
$f_*:A_M\rightarrow A_L$ defined on generating sets by $f_*(v)=f(v)$.  

There is a good model of the Eilenberg-Mac~Lane space $K(A_L,1)$, the
Salvetti complex, which we shall denote by $\ttt_L$.  For $v\in L^0$,
let $\ttt_v$ be a copy of the unit circle, viewed as a CW-complex with
one 0-cell and one 1-cell.  For each simplex $\sigma$ of $L$, define
$\ttt_\sigma$ to be the product $\prod_{v\in \sigma}\ttt_v$.  This
gives a functor from the simplices of $L$ (including the empty set,
viewed as the unique $-1$-simplex) to CW-complexes and cellular maps,
and the complex $\ttt_L$ is the colimit of this functor, which is a
subcomplex of the product $\prod_{v\in L^0}\ttt_v$.  Equivalently,
$\ttt_L$ is the polyhedral product of the pair $(\ttt,*)^L$.

The map $L\rightarrow \ttt_L$ can also be made strictly functorial,
using the group structure on the circle, where we insist that the
0-cell in the CW-complex $\ttt$ is the identity element of the group.
If $f:M\rightarrow L$ is a simplicial map of flag complexes, the group
structure on the torus is used to define a based map
$\ttt(f):\ttt_M\rightarrow \ttt_L$ that induces $f_*:A_M\rightarrow
A_L$ on fundamental groups.  To describe $\ttt(f)$, view $\ttt_M$ as a
subcomplex of the torus $\ttt^{M^0}$ and similarly, view $\ttt_L$ as a
subcomplex of the torus $\ttt^{L^0}$.  With this notation, the map
$\ttt(f)$ can be defined coordinatewise.  For $v\in L^0$, let
$U=f^{-1}(v)\subseteq M^0$.  Now the $v$-coordinate of $\ttt(f)$ takes
$(t_1,\ldots,t_k)\in \ttt^U=\prod_{u\in U}\ttt_u$ to the product
$t_1\cdots t_k\in \ttt_v$.

The Bestvina-Brady group $BB_L$ associated to a non-empty flag
complex is the kernel of the homomorphism $A_L\rightarrow \zz$
that sends each vertex to $1\in (\zz,+)$.  Like $A_L$, this is
functorial in the non-empty flag complex $L$; in particular,
if $*$ denotes a 1-vertex complex then $BB_L$ may be viewed as
the kernel of the map $A_L\rightarrow  A_*$ induced by the
unique (simplicial) map $L\rightarrow *$.  There is a good
model for $K(BB_L,1)$, defined as the infinite cyclic
covering $\widetilde{\ttt}_L$ of $\ttt_L$, which comes
equipped with a $\zz$-equivariant map to the universal
cover of $\ttt_*$, which is a copy of $\rr$.

The complex $\ttt_L$ has a single vertex.  The link of this
vertex is the sphericalization or octahedralization $\sss(L)$ of $L$.  
It has two vertices $v^+,v^-$ for each vertex $v\in L^0$, where
for any choices of signs $\epsilon_i$, the vertices
$v_0^{\epsilon_0},\ldots,v_n^{\epsilon_n}$ span an $n$-simplex of
$\sss(L)$ if and only if the vertices $v_0,\ldots,v_n$ span
an $n$-simplex of $L$.  In the case when $L$ is itself an
$n$-simplex, $\sss(L)$ is an $n$-sphere, triangulated as the
boundary of the $(n+1)$-dimensional analogue of the octahedron.  
The universal covering $X_L$ of $\ttt_L$
is a CAT(0) cubical complex, on which $A_L$ acts freely cellularly,
wtih one orbit of vertices.  There is an $A_L$-equivariant map
$X_L\rightarrow X_*\cong \rr$ which we view as a height function
on $X_L$; the subgroup $BB_L$ is the subgroup of elements that
act trivially on $X_*$, while each of the standard generators
for $A_L$ acts on $X_*\cong \rr$ as translation by~1.

Provided that $L$ is connected, the group $BB_L$ is generated by
elements indexed by the directed edges of $L$, where the directed edge
$a$ from $x$ to $y$ corresponds to the element $x^{-1}y$ of $BB_L\leq
A_L$.  The two directions of a directed edge correspond to mutually
inverse elements, and for each directed cycle $(a_1,\ldots,a_l)$ and
each integer $n$, the product $a_1^na_2^n\cdots a_l^n$ is the
identity.  It can be shown that these relators, for all cycles and all
non-zero integers~$n$, suffice to present $BB_L$~\cite{dicksleary}.
For directed cycles $(a,b,c)$ of length~3, the relators
$abc=1=a^{-1}b^{-1}c^{-1}$ imply that $a$, $b$~and~$c$ commute and
generate a group isomorphic to $\zz^2$, so for cycles of length~3,
only the relators for $n=\pm 1$ are needed.  See~\cite{dicksleary} for
more details.

Now suppose that $M$ is a connected regular covering of $L$, with
$\pi=\pi(M,L)$ as its group of deck transformations.  In this case
$\pi$ acts on the given presentations for $A_M$ and $BB_M$ by permuting
the generators and the relations.  Thus we can form the semi-direct
products $A_M\semi \pi$ and $BB_M\semi\pi$.  The action of $\pi$ on
$M$ also induces an action of $\pi$ on $\ttt_M$, which permutes the
cells freely except that the vertex is fixed.  In this way $A_M\semi\pi$
is realized geometrically as the group of all self-isomorphisms of
$X_M$ that lift the action of some element of $\pi$ on $\ttt_M$.
If $H$ is a subgroup of $A_M\semi\pi$ that maps isomorphically to $\pi$
under the map $A_M\semi\pi\rightarrow \pi$, then $H$ fixes a vertex of $X_M$,
and no other point of $X_M$.  This follows from the facts that the
action is by isometries, so the geodesic between two fixed points would
also be fixed, and that the group $H$ acts freely on the link of the
vertex that it fixes.  Thus in $A_M\semi\pi$, there is a bijective
correspondence between the vertices of $X_M$ and the subgroups $H$
that map isomorphically to $\pi$.  Furthermore, each such $H$ is
its own normalizer, because the normalizer of $H$ must act on the
fixed point set $X_M^H$ but $H$ is the entire stabilizer of this
set.  Since there is just one $A_M$-orbit of vertices in $X_M$,
all of these subgroups are conjugate in $A_M\semi\pi$.

Now consider the group $BB_M\semi\pi$.  Under the action of
$BB_M\semi\pi$, vertices of different heights lie in different orbits,
while vertices of the same height lie in the same orbit.  
Hence one sees that the conjugacy classes in $BB_M\semi\pi$ of
vertex stabilizers are permuted freely transitively by $A_M/BB_M\cong\zz$.
Choosing for once and for all an equivariant bijection between the
set of conjugacy classes of vertex stabilizers in $BB_M\semi\pi$
and the group $A_M/BB_M\cong \zz$, we can index the conjugacy classes
of vertex stabilizers by $\zz$.  (Equivalently, this amounts to
fixing a choice of splitting map $\pi\rightarrow BB_M\semi\pi$.)  
For $S\subseteq \zz$, let $N(S)$
denote the normal subgroup of $BB_M\semi\pi$ generated by the
stabilizers of the vertices whose height lies in $S$.  The group
$G_L^M(S)$ can be defined as the factor group $BB_M\semi\pi/N(S)$.
A geometric argument (essentially~\cite[lemmas~14.3,14.4]{ufp})
shows that if $H$ is the stabilizer of a vertex of height not in
$S$, then $H\cap N(S)$ is trivial.  The quotient complex $X_M/N(S)$
has vertex links $\sss(L)$ for the vertices of height in $S$ and
$\sss(M)$ for the vertices of height not in $S$.  Since $N(S)$ is
generated by elements that fix a vertex of $X_M$, the quotient
complex $X_M/N(S)$ is simply connected.  Hence by Gromov's criterion
it is CAT(0), and we define it to be $X_L^M(S)$.  The group $G_L^M(S)$
acts on it by isometries, freely except that vertices whose height is
not in $S$ have stabilizer isomorphic to $\pi$.

Since the conjugate $vN(S)v^{-1}$ of $N(S)$ by any $v\in M^0$ is equal
to $N(S+1)$, for general $S$ the subgroup $BB_M$ is the entire
normalizer of $N(S)$ inside $A_M$.  The only exceptions to this are
subsets $S$ that are periodic: if $S+n=S$ for some $n>0$, then the
normalizer of $N(S)$ contains the index $n$ subgroup of $A_M$ which is
the inverse image in $A_M$ of the unique index $n$ subgroup of
$A_M/BB_M\cong \zz$.  If we write $H$ for this subgroup, then the
quotient $H/N(S)$ is a (necessarily split) extension with kernel
$G_L^M(S)$ and infinite cyclic quotient, which we will denote by
$G_L^M(S)\semi n\zz$.  Since $H$ acts cocompactly on $X_M$ (with $n$
orbits of vertices), it follows that $G_L^M(S)\semi n\zz$ acts
cocompactly on $X_L^M(S)=X_M/N(S)$.  This is the group and action
that feature in the statement of Conjecture~\ref{conj:special}.

The presentation that we gave for $BB_M$ gives rise to a presentation
for each group $G_L^M(S)$.  For simplicity, we will focus mainly on
the case when $0\in S$.
Since for $v\in M^0$, we have that $vN(S)v^{-1}=N(S+1)$, the only isomorphism
type that is not covered by this assumption is $G_L^M(\emptyset)=BB_M\semi\pi$.
First consider the case $S=\{0\}$.  Killing the standard copy of $\pi$
inside $BB_M\semi\pi$ has the effect of identifying directed edges of
$M$ that lie in the same $\pi$-orbit.  Hence the group $G_L^M(\{0\})$
has the directed edges of $L$ as its generators.  The relators have a
similar form to the relators in the presentation we described above
for $BB_L$, except that we now have relators of the form
$a_1^na_2^n\cdots a_l^n=1$ for all $n\in \zz$, but only for directed
edge loops in $L$ that lift to loops in $M$ under the covering map.  
With respect to this presentation, one may readily describe a representative
of each conjugacy class of vertex stabilizers.  Fix a vertex $v\in L$ and
a lift $v_0\in M$ of $v$.  For each vertex $v_1$ in the orbit $\pi.v_0$, pick a
directed edge loop $(a_1,\ldots,a_l)$ in $L$ that can be lifted to a path
in $M$ from $v_0$ to $v_1$.  Since the composite of one such loop with
the reverse of another is a directed loop in $L$ that lifts to a loop
in $M$, the group element defined by $a_1^na_2^n\cdots a_l^n$ does not
depend on the choice of the loop, only on $v_0$ and $v_1$.  For each
fixed $n\neq 0$, these elements form a subgroup of $G_L^M(S)$ that is
isomorphic to $\pi$.  Different values of $n$ correspond to different
conjugacy classes of subgroup.

In this way, we obtain a presentation for $G_L^M(S)$ whenever $0\in S$.
The generators are the directed edges of $L$.  For each $n\in S$, we
take the relators $a_1^na_2^n\cdots a_l^n$ for all directed edge loops
$(a_1,\ldots,a_l)$.  For each $n\notin S$, we also take the relators
$a_1^na_2^n\cdots a_l^n$, but only for those directed edge loops in $L$
that lift to loops in $M$.  

This presentation makes clear the functoriality of the group $G_L^M(S)$.
Given any commutative square of simplicial maps in which the vertical
maps are coverings
\[
\begin{matrix}
  M'&\rightarrow &M\\
  \downarrow&&\downarrow\\
  L'&\rightarrow &L\\
\end{matrix}
\]
and any inclusion $S'\subseteq S\subseteq \zz$, there is an induced
homomorphism $G_{L'}^{M'}(S')\rightarrow G_L^M(S)$.  In particular
this applies when $L'\rightarrow L$ is a simplicial map and $M'$
is obtained as a connected component of the pullback of a covering
$M$ of $L$, and when $L'=L$ and $M'$ is a covering of $L$ that
factors through $M$.

For the sake of completeness, we now give some more details about the
presentation for $G_L^M(\emptyset)$ and in particular, we describe
representatives of the conjugacy classes of subgroups isomorphic to
$\pi=\pi(M,L)$.  The generators for $G_L^M(\emptyset)$ will be the
directed edges $a,b,c,\ldots$ of $M$ together with the elements
$g,h,j,\ldots$ of $\pi$.  We view $\pi$ as acting on the
\emph{left}  of $M$ via deck transformations, and for $g\in \pi$
and $a$ a directed edge of $M$, let $g\cdot a$ be the directed
edge obtained by acting on $a$ by $g$.  Thus the relations for
$G_L^M(\emptyset)=BB_M\semi \pi$ are the relations previously
described in the presentation for $BB_M$, the relations that
hold between the elements of $\pi$, and the conjugation relations
which have the form $gag^{-1}=g\cdot a$ for each $g\in \pi$ and
each directed edge $a$.  If $\gamma=(a_1,\ldots,a_l)$ is a
directed edge path in $M$, it will be convenient to introduce
the notation $\gamma[n]$ for the group element $a_1^na_2^n\cdots
a_l^n$.

To construct representatives of the different conjugacy
classes of subgroups isomorphic to $\pi$, we first fix a vertex
$v\in M$.  Next, for each $g\in \pi$, choose a path $\gamma_g$
from $v$ to $g\cdot v$.  If $\gamma'_g$ is another such choice,
then the concatenation $\gamma_g.\overline{\gamma'_g}$ of
$\gamma_g$ with the reverse of $\gamma'_g$ is a closed loop
in $M$.  From this it follows that for each $n\in \zz$,
$\gamma_g[n]=\gamma'_g[n]$, so the group element $\gamma_g[n]$
does not depend on the choice of path $\gamma$.  For each $n\in
\zz$, define a subset $\pi(n)$ of $G_L^M(\emptyset)$ as
\[\pi(n):= \{\gamma_g[n]g\,\,\:\,\,g\in \pi\}.\]
We claim that $\pi(n)$ is a subgroup of $G_L^M(\emptyset)$ that
is isomorphic to $\pi$.  Since $\pi(n)$ maps bijectively to
$\pi$ under the map $G_L^M(\emptyset)\rightarrow \pi$, this
claim will follow provided that $\pi(n)$ is closed under
multiplication.

For any
$g,h\in \pi$, the directed path $g\cdot \gamma_h$ obtained by
applying $g$ to the path $\gamma_h$ is a path from $g\cdot v$ to
$gh\cdot v$.  Hence the concatenation $\gamma_g.(g\cdot\gamma_h)$
is a directed path from $v$ to $gh\cdot v$.  From this it follows
that in $G_L^M(\emptyset)$ 
\[\gamma_g[n]g\gamma_h[n]h= \gamma_g[n](g\gamma_h[n]g^{-1})gh
= (\gamma_g.g\cdot\gamma_h)[n]gh
=\gamma_{gh}[n]gh.\]
Hence $\pi(n)$ is closed under multiplication and is a subgroup
isomorphic to $\pi$ as claimed.  Replacing the vertex $v$ by
another vertex $w\in M$ gives rise to a group that is conjugate
to the group $\pi(n)$; just conjugate by the element $\gamma[n]$,
where $\gamma$ is a path between $v$ and $w$.  To show that the
groups $\pi(n)$ represent all of the different conjugacy classes
one can use a geometric argument.  Alternatively, the description
of the conjugation action of $A_M$ on $BB_M$ in~\cite{dicksleary}
can be used to show directly that the conjugate of $\pi(n)$ by the
vertex~$v$, viewed as one of the generators for $A_M$, is equal to
$\pi(n+1)$.  

The above description of the subgroups $\pi(n)$ of $G_L^M(\emptyset)$
gives a way to present each $G_L^M(S)$ as a quotient of
$G_L^M(\emptyset)$.  To the relations for $G_L^M(\emptyset)$ one
adds the relation $\gamma_g[n]g=1$ for each $n\in S$ and each $g\in
\pi$.

Next we briefly recall some material concerning special cube complexes
from~\cite{hw}.  A {\sl hyperplane} in a non-positively curved cube complex
is an equivalence class of directed edges under the relation generated
by `form the opposite directed edges of a square'.  A hyperplane is
{\sl 2-sided} if it does not contain any pair of directed edges associated
to a single directed edge.  Hyperplanes {\sl intersect} if they contain
directed edges that are adjacent sides of a square.  Two directed
edges {\sl directly osculate} if they are not contained in a square and
share the same terminal vertex.  A hyperplane {\sl self-intersects} if there
is a square containing two directed edges of the hyperplane as
adjacent sides.  A hyperplane {\sl directly self-osculates} if it contains
two directed edges that directly osculate.  Two hyperplanes
{\sl inter-osculate} if they intersect and also contain a pair of directed
edges that directly osculate.

A locally CAT(0) cube complex is
{\sl $A$-special} if its hyperplanes are 2-sided and do not self-intersect,
directly self-osculate or inter-osculate.  The fundamental group of a
finite $A$-special cube complex embeds in a right-angled Artin group and
hence is torsion-free and linear over~$\zz$, which implies that it is
residually finite~\cite[thm.~1.1]{hw}.  A {\sl special} locally CAT(0) cube
complex is similar to an $A$-special complex except that hyperplanes
are not required to be 2-sided.  For each of our complexes $X_L^M(S)$
the height function $X_L^M(S)\rightarrow \rr$ allows one to distinguish
upward and downward pointing directed edges.  For this reason every
hyperplane in $X_L^M(S)/H$ is 2-sided for any $H\leq G_L^M(S)$, and
so $X_L^M(S)/H$ is special if and only if it is $A$-special.  

We close this section with some remarks concerning the profinite
topology on $\zz$.  This is the topology in which the basic open
sets are the periodic subsets.  Each periodic subset is also closed,
and it follows that every open set is a union of periodic sets and
that every closed set is an intersection of periodic sets.  We make
a more precise version of this second statement below.

\begin{lemma}\label{lem:closed}
  Any $S\subseteq \zz$ that is closed in the profinite topology is the
  intersection of a nested sequence $T_1\supseteq T_2\supseteq
  T_3\cdots$ of periodic sets.  Furthermore, we may suppose that
  $S\cap [-n,n]=T_n\cap [-n,n]$.
\end{lemma}

\begin{proof}
  Let $n_1,n_2,n_3,\ldots$ be the elements of $\zz-S$, enumerated so
  that $|n_i|\leq |n_j|$ whenever $i<j$.  Since $S$ is closed, for
  each $i$ we can find an open set $O_i$ with $n_i\in O_i$ and $S\cap
  O_i=\emptyset$.  Since the periodic sets form a basis for the
  topology, we may suppose in addition that $O_i$ is periodic.  If 
  we let $F_i=\zz-O_i$ then $F_i$ is periodic, $S\subseteq F_i$, and
  $n_i\notin F_i$.  From this it follows that $\bigcap_i F_i=S$.  If
  we define $T_n=\bigcap_{i=1}^{2n+1}F_i$ then $T_n$ has the
  properties claimed in the statement.
\end{proof}

The natural numbers $\nn\subset \zz$ is an easy example of a subset
that is far from being either open nor closed: the only closed set
that contains $\nn$ is $\zz$ and the only open set contained in $\nn$
is the empty set.  We give a construction of a large collection of
closed subsets, starting with a G\"odel numbering $\phi$ of the finite
subsets of $\nn$.  For $F$ a finite subset of $\nn$, define
\[\phi(F)=\sum_{n\in F} 10^n,\]
with the usual convention that the empty sum is~0.  
The image of $\phi$ is the subset $T(\nn)$ consisting of all positive
integers all of whose decimal digits are equal to 0~or~1.  Now for any
$S\subseteq \nn$, define $T(S)\subseteq\zz$ to be the image under $\phi$ of the
finite subsets of $S$; equivalently $T(S)$ is the positive integers
with digits $\{0,1\}$, where the $n$th digit is~0 when $n\notin S$.

\begin{proposition}\label{prop:tsclosed}
  For each $S\subseteq \nn$, the set $T(S)$ is closed.  The set
  $T(S)$ (resp.~$\nn-T(S)$) is recursively enumerable if and only
  if $S$ (resp.~$\nn-S$) is.
  \end{proposition}

\begin{proof}
  For $n\geq 0$, let $F_n=(T(S)\cap [0,2.10^n])+10^{n+1}\zz$.  Each
  $F_n$ is periodic and $T(S)=\bigcap_n F_n$, which implies that
  $T(S)$ is closed.  The definition of $T(S)$ gives a recursive
  procedure for computing $T(S)$ from $S$, showing that $T(S)$ is
  recursively enumerable when $S$ is.  Computing $\nn-T(S)$ from
  $\nn-S$ is slightly more complicated.  To compute $\nn-T(S)$, fix an
  integer $N$ and run for $N$ steps an algorithm to generate elements
  of $C:=\nn-S$, keeping a list $C(N)$ of the elements of $C$ so
  obtained.  Then output every integer in the range $[0,N]$ that
  either has a decimal digit not equal to 0~or~1, or has its $n$th
  decimal digit equal to~1 for some $n\in C(N)$.  Now repeat this
  procedure for increasing values of $N$.  For the converse
  statements, note that $10^n\in T(S)$ if and only if $n\in S$.  Thus
  a recursive enumeration of $T(S)$ (resp.~$\nn-T(S)$) gives rise to a
  recursive enumeration of $S$ (resp.~$\nn-S$).
\end{proof}

\section{A set-valued invariant}\label{sec:three}

In this section we prove Theorem~\ref{thm:properties} using a
set-valued invariant of a group and a sequence of elements,
$\calr(G,\bg)$ that was introduced in~\cite{ufp}.  But first
we prove Proposition~\ref{prop:torfree}.  

\begin{proof} (Proposition~\ref{prop:torfree})
  $G_L^M(S)$ acts on the CAT(0) cubical complex
  $X_L^M(S)$, freely except that vertices whose height is
  not in $S$ have stabilizer isomorphic to $\pi(M,L)$.
  Any action of an element of finite order on any CAT(0)
  space must fix a point, and the claim follows.
\end{proof}

For a group $G$ and a finite sequence $\bg=(g_1,\ldots, g_l)$ of
elements of $G$, the invariant $\calr(G,\bg)$ is the subset of
$\zz$ defined by
\[\calr(G,\bg)=\{n\in \zz\,\,:\,\,g_1^ng_2^n\cdots g_l^n=1\}.\]
In \cite[lemma~15.3]{ufp}, it was shown that in the case when $G=G_L^M(S)$
and $\bg$ is the sequence of generators spelling out an edge
loop in $L$ that does not lift to a loop in $M$, then
$\calr(G,\bg)=S$.  (This was stated only in the case when $M$ is
the universal cover, but the same proof holds in general.)  
We require another property:

\begin{proposition}\label{prop:gfinite}
  If $G$ is finite, $\calr(G,\bg)$ is periodic.  If $G$ is residually
  finite, $\calr(G,\bg)$ is closed in the profinite topology on $\zz$.  
\end{proposition}

\begin{proof}
  For the first statement, it suffices to make the 
  weaker assumption that $G$ has finite exponent, $m$ say.
  In this case for any $n$,  $g_1^{n+m}g_2^{n+m}\cdots g_l^{n+m}= g_1^ng_2^n
  \cdots g_l^n$, from which it follows that $\calr(G,\bg)=\calr(G,\bg)+m$.

  For the second statement, whenever $f:G\rightarrow Q$ is a homomorphism
  from $G$ to a finite group, the first statement implies that
  $\calr(Q,f(\bg))$ is periodic.  Since $G$ is assumed to be 
  residually finite, $g_1^n\cdots g_l^n$ is equal to the identity
  if and only if its image under each such $f:G\rightarrow Q$ is.
  Hence $\calr(G,\bg)$ is the intersection of a family of periodic
  sets of the form $\calr(Q,f(\bg))$.  This proves the claim.  
\end{proof} 

\begin{proof} (Theorem~\ref{thm:properties})
  Suppose firstly that $G_L^M(S)$ is virtually torsion-free.  
  If $S=\zz$ or $\pi(M,L)$ is torsion-free, then we have already seen
  that $G_L^M(S)$ is torsion-free.  Thus we may suppose that $S\neq \zz$
  and that $\pi(M,L)$ contains some torsion.  Since $S\neq
  \zz$, $G_L^M(S)$ contains subgroups isomorphic to $\pi(M,L)$, and so
  $\pi(M,L)$ must be virtually torsion-free as claimed.  
  Now let $\gamma=(a_1,\ldots,a_l)$ be
  a directed edge loop in $L$ that represents a non-trivial torsion element
  in $\pi(M,L)$.  Thus $\gamma$ does not lift to a closed loop in
  $M$ but there is some $m>1$ the iterated loop $\gamma^m$ does
  lift to a closed loop in $M$.  Now suppose that $f:G_L^M(S)\rightarrow Q$
  is a homomorphism to a finite group with torsion-free kernel.  For each
  $n$, $a_1^n\cdots a_l^n$ is a torsion element of $G_L^M(S)$, and this
  element is equal to the identity if and only if $n\in S$.  Since the
  kernel of $f$ is torsion-free, it follows that
  $\calr(Q,(f(a_1),\ldots,f(a_l))=S$, and so by
  Proposition~\ref{prop:gfinite}, $S$ must be periodic.  

  Next suppose that $G_L^M(S)$ is residually finite.  If either
  $S=\zz$ or $\pi(M,L)$ is trivial, then $G_L^M(S)$ is the
  Bestvina-Brady group $BB_L$, which is residually finite.  Thus
  we may assume that $\pi(M,L)\neq \{1\}$ and that $S\neq \zz$.
  Since $S\neq \zz$, $G_L^M(S)$ contains subgroups isomorphic to
  $\pi(M,L)$ and so $\pi(M,L)$ must be residually finite as claimed.
  Now let $(a_1,\ldots,a_l)$ be a directed edge loop in $L$ that
  does not lift to a loop in $M$.  The element $a_1^n\cdots a_l^n
  \in G_L^M(S)$ is equal to the identity if and only if $n\in S$.
  Hence $S=\calr(G_L^M(S),(a_1,\ldots,a_l))$ must be closed in
  the profinite topology by Proposition~\ref{prop:gfinite}.

  \end{proof} 
  
\section{Reduction to finite covers and periodic sets}\label{sec:four} 

The two theorems in this section,
Theorems~\ref{thm:fourone}~and~\ref{thm:fourtwo}, together imply
Theorem~\ref{thm:reduction}.

\begin{theorem}\label{thm:fourone}
  Suppose that $S$ is closed in the profinite topology.  For any
  non-identity element $g\in G_L^M(S)$ there is periodic $T\supseteq
  S$ so that the image of $g$ in $G_L^M(T)$ is not the identity.
\end{theorem}

\begin{proof}
  This argument is based on one in~\cite{kls}.  If $S$ is periodic
  there is nothing to prove.  If not, then $S\neq \emptyset$ and
  we may assume without loss of generality that $0\in S$, and we
  may take as generators for $G_L^M(S)$ the directed edges of $L$. 
  By~Lemma~\ref{lem:closed} there is a
  nested sequence $T_1\supseteq T_2\supseteq \cdots$ of periodic sets
  with intersection $S$ and such that $T_n\cap [-n,n]=S\cap[-n,n]$.
  It will suffice to show that the word length of any non-identity
  element of $G_L^M(S)$ that maps to the identity in $G_L^M(T_n)$
  tends to infinity with $n$.

  Since $0\in S$, the group $G_L^M(S)$ acts freely on the vertices of
  height~$0$ in $X_L^M(S)$, and the Cayley graph for $G_L^M(S)$ embeds
  in the height $0$ subset of $X_L^M(S)$, with each generator mapping
  to the diagonal of a square.  Fix $v_o$ a vertex of height $0$ in
  $X_L^M(S)$ and let $\gamma$ be the geodesic arc from $v_0$ to
  $gv_0$, and let $f_n$ denote the map $f_n:X_L^M(S) \rightarrow
  X_L^M(T_n)$.  if $f_n\circ\gamma$ is a geodesic arc in $X_L^M(T_n)$,
  then $f_n(gv_0)\neq f(v_0)$, which implies that $f_n(g)\neq 1$.
  This will happen unless $\gamma$ passes through a vertex of
  $X_L^M(S)$ of height in $T_n-S$.  By the argument used
  in~\cite[lemma~3.2]{kls}, this implies that the word length of $g$ is
  strictly greater than $n\sqrt{2/(d+1)}$, where $d$ is the dimension
  of $L$.
\end{proof} 

\begin{theorem}\label{thm:fourtwo}
  If $\pi(M,L)$ is residually finite, then for any non-identity
  element $g\in G_L^M(S)$, there is a finite regular cover
  $N\rightarrow L$ lying between $M$ and $L$ so that the image of $g$
  in $G_L^N(S)$ is not the identity.
\end{theorem}

\begin{proof}
  Let $x$ be a point of $X_L^M(S)$ such that $d(x,gx)$ is minimized.
  If $d(x,gx)=0$, then $x$ must be a vertex, and $g$ is contained in
  a conjugate of $\pi(M,L)\leq G_L^M(S)$.  By hypothesis there is a
  finite quotient $Q$ of $\pi(M,L)$ in which the image of $g$ is
  not the identity, and we may choose $N$ so that $\pi(N,L)=Q$.  

  Otherwise, let $\gamma$ be the unique geodesic arc from $x$ to
  $gx$ in $X_L^M(S)$.  If $N\rightarrow L$ is any finite regular
  covering so that $M$ also covers $N$, let $f:X_L^M(S)\rightarrow
  X_L^N(S)$ be the induced map of CAT(0) cubical complexes.  If
  $f\circ\gamma$ is a geodesic arc in $X_L^N(S)$, then $f(x)\neq f(gx)$,
  indicating that the image of $g$ in $G_L^N(S)$ is not the identity.
  Thus it suffices to show that we can choose $N$ so that $f\circ \gamma$
  is a geodesic arc.  For any $N$, the map $f$ will be a local isometry
  except at the vertices with height in $S$, so if the interior of $\gamma$
  contains no such vertices, we may take $N=L$.  In any case, there are only
  finitely many such vertices.

  For each vertex $v$ that is contained in the interior of $\gamma$,
  the inward and outward pointing parts of $\gamma$ define a pair of
  points $\gamma^-,\gamma^+\in \link_X(v) \cong \sss(M)$ necessarily
  separated by at least $\pi$ in $\sss(M)$.  Now $f\circ\gamma$ is
  locally geodesic at $f(v)$ provided that the distance in
  $\link_{f(X)}(f(v))\cong\sss(N)$ between $f(\gamma^-)$~and
  $f(\gamma^+)$ is also at least $\pi$.  The open ball of radius $\pi$
  in $\sss(M)$ centred at $\gamma^-$ contains only finitely many points
  $h\gamma^+$ of the orbit $\pi(M,L)\gamma^+$, and none of these $h$
  is the identity.  Since there are only finitely many vertices on
  $\gamma$, we obtain a finite set $\{h_1,\ldots,h_m\}$ of
  non-identity elements of $\pi(M,L)$ with the property that
  $f\circ\gamma$ is a geodesic arc provided that $f(h_i)\neq 1\in
  \pi(N,L)$ for each $h_i$.  Since $\pi(M,L)$ is residually finite, we
  can find such an $N$.
\end{proof}

\section{Simplicial approximations} \label{sec:five}

If $L'$ is a subdivision of a flag complex $L$, a simplicial
approximation to the identity is a simplicial map $f:L'\rightarrow L$
such that the induced map of topological spaces $f_*:|L'|\rightarrow |L|$
is homotopic to the identity map.  If $M\rightarrow L$ is a covering,
and $M'\rightarrow L'$ is the induced covering of $L'$, then any
simplicial approximation to the identity for $L$ will lift to a
$\pi(M,L)$-equivariant simplicial approximation to the identity
for $M$.   

\begin{definition}
  A subdivision $L'$ of $L$ is \emph{suitable} if there is a simplicial
  approximation $f$ to the identity such that for any pair $u,v$ of
  adjacent vertices of $L'$, the image $f(\Star(u)\cup \Star(v))$ is
  contained in a single simplex of $L$.
\end{definition}

\begin{proposition}
  The second barycentric subdivision of any simplicial complex is
  suitable.
\end{proposition}

\begin{proof}
  The vertices of the barycentric subdivision of $L$ are indexed by the
  simplices of $L$, with an edge joining the vertices $\tau,\sigma$ if
  and only if one of $\tau$ and $\sigma$ is a face of the other.
  There is a well-known description of a simplicial approximation in
  this case: fix a partial order on the vertex set of $L$ that is
  total when restricted to each simplex, and send the vertex $\sigma$
  of the barycentric subdivision to the least vertex in $L$ of the
  simplex $\sigma$.

  The vertices of the second barycentric subdivision of $L$ are
  indexed by chains $\underline{\sigma} = \sigma_0<\sigma_1<\cdots
  <\sigma_n$ of simplices of $L$, where $\underline{\sigma}$ and
  $\underline{\tau}$ are joined by an edge if and only if one is a
  subchain of the other.  Write $\underline{\tau}\subseteq
  \underline{\sigma}$ to indicate that $\underline{\tau}$ is a
  subchain of $\underline{\sigma}$.  The natural choice of a partial
  order on the vertices of the barycentric subdivision is to order
  by dimension of the corresponding simplex of $L$.  With these
  choices, the composite simplicial approximation from the
  second barycentric subdivision to $L$ sends the vertex
  $\underline{\sigma}$ to the least vertex of the simplex
  $\sigma_0$, i.e., the least vertex of the minimal simplex in
  the chain.

  A vertex in the star $\Star(\underline{\sigma})$ is
  either a subchain or a superchain of $\underline{\sigma}$.  If
  $\underline{\tau}\subseteq \underline{\sigma}$, then the minimal
  simplex $\tau_0$ of $\underline{\tau}$ is contained in
  $\sigma_n$, the maximal simplex
  of $\underline{\sigma}$.  If instead $\underline{\tau}\supseteq
  \underline{\sigma}$, then $\tau_0$
  is contained in $\sigma_0$.  In either case, the minimal vertex
  of the minimal simplex of $\underline{\tau}$ is a vertex of
  $\sigma_n$, and so $f(\Star(\underline{\sigma}))$ is contained
  in the simplex $\sigma_n$.

  If there is an edge in the second barycentric subdivision
  between $\underline{\sigma}$ and $\underline{\tau}$, then
  one of the two chains is a subchain of the other, so we may 
  suppose $\underline{\tau}\subseteq \underline{\sigma}$.
  In this case the maximal simplex $\tau_p$ is contained
  in the maximal simplex $\sigma_n$, and so
  $f(\Star(\underline{\tau})\cup\Star(\underline{\sigma}))$
  is contained in the single simplex $\sigma_n$ of $L$.
\end{proof}

\section{Special cube complexes}\label{sec:six}

There is a natural identification of $X_L^M(S)/G_L^M(S)$ with
$X_L/BB_L$, so we start by considering the cube complex $X_L/BB_L$,
and its quotient $\ttt_L=X_L/A_L$.  The Salvetti complex $\ttt_L$
has one vertex, edges in bijective correspondence with the vertex
set $L^0$ and squares in bijective correspondence with $L^1$.
The map $\ttt_L\rightarrow \ttt$ described earlier lifts to a
map $X_L/BB_L\rightarrow \rr$ such that the image of each vertex
is an integer, and furthermore the image of each $n$-cube of $X_L$
is an interval of length~$n$.

We may view edges of $X_L^M(S)$ as being labelled by elements of
$L^0$ via the identification of $X_L^M(S)/G_L^M(S)=X_L/BB_L$, and
the squares of $X_L^M(S)$ as being labelled by elements of $L^1$.
The function $X_L/BB_L\rightarrow \rr$ induces a height function
on $X_L^M(S)$.  This height function and the labellings discussed
above are preserved by the action of $G_L^M(S)$.  In $X_L/BB_L$
there is one vertex of each height $n\in \zz$, and for each
$x\in L^0$ there is one edge labelled $x$ whose vertices are
of heights $n$~and~$n+1$.  Directed edges either point upwards
or downwards, and the opposite sides of a square of $X_L^M(S)$
point the same way.

If $H$ is a finite index normal subgroup of $G_L^M(S)$, it follows
that $X_L^M(S)/H$ has finitely many vertices of each height, and
finitely many edges of each height.  Moreover, the group $G_L^M(S)/H$
acts freely and transitively on the edges with label $x\in L^0$
of each fixed height.  This group acts transitively on the vertices
of each fixed height too, but for vertices whose height is not in
$S$, the stabilizer of a vertex is the group $\pi(M,L)/(\pi(M,L)\cap H)$.

Since the adjacent sides of each square have distinct labels in $L^0$,
no hyperplane of $X_L^M(S)/H$ can self-intersect.  Since the opposite
sides of each square point either upwards or downwards, no hyperplane
in $X_L^M(S)/H$ can fail to be 2-sided.  Thus to establish that
$X_L^M(S)/H$ is special, we only need to check that there are no
direct self-osculations and no inter-osculations.

Each simplex $\sigma$ of $L$ corresponds to a coordinate subtorus
$\ttt_\sigma$ of $\ttt_L$, and this lifts to a single infinite
cylinder inside $X_L/BB_L$.  If we identify the torus $\ttt_\sigma$
with the quotient $\rr^{n+1}/\zz^{n+1}$, where $\sigma$ is an
$n$-simplex, then its preimage inside $X_L/BB_L=X_L^M(S)/G_L^M(S)$
is the quotient $\rr^{n+1}/K$, where $K=\{(m_0,\ldots,m_n)\in \zz^{n+1}:
m_0+\cdots m_n=0\}$, a subgroup of $\zz^{n+1}$ of rank $n$.

The link of a vertex of $X_L^M(S)$ is either $\sss(L)$ for a vertex of
height in $S$ or $\sss(M)$ for a vertex of height not in $S$.  If
$\sigma$ is any simplex of $L$, then the inverse image of $\sigma$ in
$M$ is a disjoint union of finitely many simplices
$\sigma_1,\ldots,\sigma_k$ of $M$, where $k$ is the index of the
cover.  The inverse image of $\sss(\sigma)$ in $\sss(M)$ is thus a
disjoint union of $k$ copies of the $n$-sphere
$\sss(\sigma_1)\sqcup\ldots,\sqcup\,\sss(\sigma_k)$.  

To simplify the discussion, suppose from now on that $\pi(M,L)$ is
finite and that $H$ is torsion-free.  In this case, the stabilizer
in $G_L^M(S)/H$ of each vertex of $X_L^M(S)/H$ of height not in $S$
is isomorphic to $\pi(M,L)$.  By the observations above,
the inverse image in $X_L^M(S)/H$ of the cylinder of $X_L/BB_L$
labelled by $\sigma$ is a disjoint union of finite covers of the
cylinder, except that the vertices of the inverse image of height
not in $S$ are identified in orbits of size $\pi(M,L)$.  The cylinders
labelled by a given simplex are permuted by $G_L^M(S)/H$, and the
stabilizer of each cylinder in this action is abelian and generated by
at most $n$ elements: if the cylinder
is $\rr^{n+1}/K$, and we view $K$ as a subgroup of $G_L^M(S)$, then
the stabilizer is the group $K/(K\cap H)\cong KH/H\leq G_L^M(S)/H$.

\begin{lemma}\label{lem:cylint}
  If the intersection of two cylinders contains an edge $e$ with a
  given label in~$L^0$, then it contains an edge of each height with
  that same label.
\end{lemma}

\begin{proof}
  If $v$ is one of the vertices of the edge $e$, $e$ defines a vertex
  of the link $\link_X(v)$ which is either $\sss(L)$ or $\sss(M)$.
  By~\cite[prop.~7.3]{ufp} the antipode of this point of $\link_X(v)$
  is uniquely determined, and corresponds to an edge $e'$ of
  $X_L^M(S)$ with the same label as $e$ but with height differing by
  one from that of $e$.  If $e$ is contained in a cylinder $C$, then
  so is $e'$.
\end{proof}

\begin{definition}
  Say that edges $e$, $e'$ of $X_L^M(S)/H$ are \emph{cylinder equivalent} if
  they have the same label in $L^0$ and there are $r\geq 0$, edges
  $e_0,\ldots,e_r$ and cylinders $C_1,\ldots,C_r$ so that each $e_i$
  has the same label as $e$, and for $1\leq i\leq r$ both
  $e_{i-1}$~and~$e_i$ are contained in the cylinder $C_i$.
\end{definition}

Let $Q=G_L^M(S)/H$ be the group of deck transformations of the
branched cover $X_L^M(S)/H\rightarrow X_L/BB_L$, so that $Q$ acts
freely on the edges of $X_L^M(S)/H$ and permutes the cylinders.  

\begin{proposition}
  Suppose that $e'$ is an edge of the same height as $e$, and let
  $P\leq Q$ be the subgroup generated by the stabilizers of all
  of the cylinders that contain $e$.  Then $e'$ is cylinder equivalent
  to $e$ if and only if $e'$ lies in the orbit $Pe$.
\end{proposition}

\begin{proof}
  By induction on the length $r$ of the chain of cylinders used to
  establish the cylinder equivalence.  By Lemma~\ref{lem:cylint},
  if $e'$ is cylinder equivalent to $e$, we may choose
  $e_0,\ldots,e_r$ and $C_1,\ldots,C_r$ as in the definition with
  each $e_i$ having the same height as $e$.  Let $P_1$ be the
  stabilizer in $Q$ of $C_1$.  By transitivity, there exists
  $g\in P_1$ so that $ge=ge_0=e_1$.  Since there is a shorter
  cylinder equivalence from $e_1$ to $e_r$, there exists $g'\in P$
  so that $g'e_1=e_r=e'$, and since $P_1\leq P$, we see that $g'g\in
  P$ and that $(g'g)e=e'$.  This argument can be reversed, giving
  the converse.
\end{proof}

\begin{proposition}
  If $e$ and $e'$ are in the same hyperplane, then $e$ and $e'$ are
  cylinder equivalent.
\end{proposition}

\begin{proof}
  Each square of $X_L^M(S)/H$ is contained in a cylinder.
\end{proof}

\begin{proposition}\label{prop:vspec} 
  Suppose that $S+n=S$ and that $M\rightarrow L$ is a finite regular
  cover.  If $G_L^M(S)$ is virtually (non-cocompact) special then
  $G_L^M(S)\semi n\zz$ is virtually special.
\end{proposition}

\begin{proof}
  Suppose that $H\leq G_L^M(S)$ is a finite-index subgroup such that
  $X_L^M(S)/H$ is a special cube complex.  Since $G_L^M(S)$ contains
  finitely many subgroups of a given index, by passing to a subgroup
  if necessary we may assume that $H$ is 
  characteristic in $G_L^M(S)$, so that $n\zz$ normalizes $H$.  The
  action of $n\zz$ on $X_L^M(S)/H$ identifies vertices of different
  heights, so it does not create any new pairs of edges where an
  osculation takes place.  However, it may be that the compact cube
  complex $X_L^M(S)/(H\semi n\zz)$ has fewer hyperplanes than
  $X_L^M(S)/H$, which may cause extra interosculations or
  self-osculations.  To avoid this, note that $n\zz$ acts as
  permutations of the finitely many hyperplanes in $X_L^M(S)/H$,
  and so for some $m>0$ the subgroup $mn\zz\leq n\zz$ preserves
  each hyperplane.  For this $m$, $X_L^M(S)/(H\semi mn\zz)$ will
  be special because $X_L^M(S)/H$ is by hypothesis.
\end{proof}

\begin{theorem}\label{thm:bothvspec}
  Suppose that $M\rightarrow L$ is a finite cover, and that
  $\theta:G_L^M(S)\rightarrow Q$ is a homomorphism to a finite
  group such that the kernel of $\theta$ is torsion-free and
  such that for any two adjacent vertices $u,v\in L$, the image
  under $\theta$ of the subgroup of $G_L^M(S)$ generated by the
  edges of $\Star(u)\cup\Star(v)$ is abelian.  Then $G_L^M(S)$
  is (non-cocompact) virtually special.
\end{theorem}

\begin{proof}
  We construct a homomorphism to a larger finite group whose kernel
  will be shown to be special.
  The abelianization $H_1(A_L;\zz)$ of the right-angled Artin
  group $A_L$ is free, with basis the set $L^0$, and the
  abelianization $H_1(BB_L;\zz)$ of $BB_L$ is the codimension one summand
  consisting of all elements $\sum_v n_v v$ with $\sum_v n_v=0$.
  Now let $m$ be the exponent of the finite group $Q$, and let
  $H=H_1(BB_L;\zz/m\zz)=H_1(BB_L;\zz)\otimes \zz/m\zz$.  The
  quotient map $\phi:G_L^M(S)\rightarrow G_L^M(\zz)=BB_L\rightarrow H$
  and $\theta$ together give a homomorphism 
  $(\theta,\phi):G_L^M(S)\rightarrow Q\times H$, and it is this map whose
  kernel $K$ will be shown to be special.  

  Note that every vertex stabilizer $\pi(M,L)$ is in the kernel of the
  map from $G_L^M(S)$ to $BB_L$, and so every copy of $\pi(M,L)$ is
  mapped by $(\theta,\phi)$ to a subgroup of $Q\times \{0\}$.

  As remarked earlier, hyperplanes in $X_L^M(S)/K$ are always 2-sided
  and never self-intersect, so we only need to rule out direct
  self-osculations and inter-osculations.  Suppose that $e$ and $e'$
  are adjacent edges of $X_L^M(S)/K$ of the same height that share a
  square, so that the hyperplanes they belong to intersect, and let
  $x$ be the vertex that is incident on both $e$ and $e'$.  We claim
  that no other edge of the same height that is cylinder equivalent to
  $e$ or to $e'$ can be incident on $x$.  Since hyperplanes are
  contained in cylinder equivalence classes, this will imply that
  there are no direct self-osculations or inter-osculations.

  To establish this claim, note that the labels in $L^0$ attached
  to $e$ and $e'$ are adjacent vertices $u,v$.  

  The stabilizer of an $n$-cylinder of $X_L^M(S)$ in $G_L^M(S)/K$
  is an abelian group.  If the cylinder is labelled by an $n$-simplex
  $\sigma$, then its stabilizer is the image in $G_L^M(S)/K$ of the
  free abelian subgroup $BB_\sigma\leq G_L^M(S)$ which is of rank
  $n$.  A cylinder of $X_L^M(S)/K$ that contains the edge $e$ 
  corresponds to a simplex $\tau$ of $L$ that contains $u$ and
  similarly, a cylinder that contains $e'$ corresponds to a simplex
  $\tau'$ of $L$ that contains $v$.  Any such $\tau,\tau'$ are
  contained in the subcomplex $J=\Star(u)\cup\Star(v)$ of $L$.
  Since $J$ is simply connected, the subgroup of $G_L^M(S)$
  generated by the edges of $J$ is isomorphic to $BB_J$, and
  hence the inclusion $J\rightarrow L$ induces a monomorphism
  $H_1(J;\zz/m\zz)\rightarrow H=H_1(BB_L;\zz/m\zz)$.  By hypothesis
  $\theta(BB_J)$ is an abelian subgroup of $Q$, necessarily of
  exponent dividing $m$.  But $\phi(BB_J)$ is the largest possible
  abelian quotient of $BB_J$ of exponent $m$, and so it follows
  that the image of $BB_J$ under $(\theta,\phi)$ has trivial
  intersection with $Q\times\{0\}$.  

  The claim now follows, since any element of
  $(\theta,\phi)(G_L^M(S))\leq Q\times H$ that
  fixes the vertex $x$ must lie in $Q\times \{0\}$, whereas
  any element that sends either $e$ or $e'$ to a cylinder
  equivalent edge must lie in $(\theta,\phi)(BB_J)$.
\end{proof}

\begin{proof} (Theorem~\ref{thm:vtorfreeimpliesvspec}.)
  The kernel of the map $G_{L'}^{M'}(S)\rightarrow G_L^M(S)$ is
  torsion-free, and so if $G_L^M(S)\rightarrow Q$ is any homomorphism
  with torsion-free kernel, then the composite
  $G_{L'}^{M'}(S)\rightarrow Q$ also has torsion-free kernel.  
  Since for any two adjacent vertices $u,v$ of $L'$, the image
  of $J=\Star(u)\cup\Star(v)$ is contained in a single simplex
  of $L$, the image in $G_L^M(S)$ of the subgroup $BB_J$ is abelian
  and hence so is its image in $Q$.  Since $G_{L'}^{M'}(S)$ is
  virtually special and $S+n=S$, Proposition~\ref{prop:vspec} implies that
  $G_{L'}^{M'}(S)\semi n\zz$ is virtually special.  
  \end{proof}

\begin{corollary}\label{cor:abelian}
  Suppose that $M$ is a finite cover of $L$ and that 
  there is a homomorphism $G_L^M(S)\rightarrow Q$
  with torsion-free kernel, with $Q$ is a finite abelian group.
  Then $G_L^M(S)$ is virtually (non-cocompact) special, and if
  $S+n=S$ then $G_L^M(S)\semi n\zz$ is virtually special.
\end{corollary}

\begin{proof}
  Follows from Theorem~\ref{thm:bothvspec} and
  Proposition~\ref{prop:vspec}.  
\end{proof}

\section{Covers pulled back from graphs} \label{sec:seven}

In this section we prove Theorem~\ref{thm:vtorfree}.  We start with a
Proposition in finite group theory.  Let $A_N$ and $S_N$ denote the
alternating and symmetric groups on a finite set of size $N$, let
$C_n$ denote a finite cyclic group of order $n$, The wreath product
$S_N\wr C_n$ is a finite group containing a normal subgroup isomorphic
to the direct product $(S_N)^n$ of $n$ copies of $S_N$, with quotient
the cyclic group $C_n$.  
If $\rho$ is a generator for the cyclic group $C_n$, conjugation by
$\rho$ permutes the $n$~copies of $S_N$ freely.  If
$\alpha$ denotes a permutation in
$S_N$, we write $\alpha_i$ with $1\leq i\leq n$ for the element of
the product $(S_N)^n$ that has its $i$th component equal to $\alpha$
and the other components equal to the identity.  In the usual notation 
for elements of a direct product, $\alpha_i$ would be written
as $(1,\ldots,1,\alpha,1,\ldots,1)$.  For permutations $\alpha,\beta$,
the elements $\alpha_i$ and $\beta_j$ commute if $i\neq j\in \zz/n\zz$,
while $\alpha_i\beta_i = (\alpha\beta)_i$, and the conjugation action
of $\rho$ is given by
   \[
   \rho\alpha_i\rho^{-1}=\begin{cases}
   \alpha_{i+1} &i<n,\\
   \alpha_1 & i=n.
   \end{cases}
  \] 

  \begin{proposition}\label{prop:hom}
    Let $\alpha,\beta$ be elements of $S_N$ and let
    $\sigma$ be the commutator 
    $\alpha\beta\alpha^{-1}\beta^{-1}$.  For $1\leq k<n$ define
    elements $a,b,c,d\in S_N\wr C_n$, depending on $k$ as well as
    $N$~and~$n$, by the formulae 
     \[a= \rho,\quad
     b= \alpha_n\rho^{-1}\alpha_n^{-1}, \quad
     c= \alpha_n\beta_k\rho\beta_k^{-1}\alpha_n^{-1},\quad
     d= \beta_k\rho^{-1}\beta_k^{-1}.\]
     For each $k$, the elements $a,b,c,d$ all have order $n$,
     and for any integer $j$,
     \[a^jb^jc^jd^j = \begin{cases}
       \sigma_k& j\equiv k \,\,\hbox{modulo}\,\,n,\\
       1& j\not\equiv k \,\,\hbox{modulo}\,\,n. 
     \end{cases}\] 
   \end{proposition}

   \begin{proof}
     Each of $a,b,c,d$ is a conjugate of
     either $\rho$ or $\rho^{-1}$, so each has order $n$ as
     claimed.  From this is follows that it suffices to check
     the claim concerning the order of $a^jb^jc^jd^j$ for $1\leq j<n$.
     For $j$ in this range, since $\alpha_n$ commutes with $\beta_k$
     we see that 
     \begin{align*} 
     a^jb^jc^jd^j&= \rho^j(\alpha_n\rho^{-j}\alpha_n^{-1})
     (\alpha_n\beta_k\rho^j\beta_k^{-1}\alpha_n^{-1})
     (\beta_k\rho^{-j}\beta_k^{-1})\\
     &=\rho^j\alpha_n\rho^{-j}(\alpha_n^{-1}\alpha_n\beta_k)\rho^j
     (\beta_k^{-1}\alpha_n\beta_k)\rho^{-j}\beta_k^{-1}\\
     &=(\rho^j\alpha_n\rho^{-j})
     \beta_k(\rho^j\alpha_n^{-1}\rho^{-j})\beta_k^{-1}\\
     &= \alpha_j\beta_k\alpha_j^{-1}\beta_k^{-1}=\begin{cases}
     \sigma_k& j=k\\
     1&j\neq k.  
     \end{cases} 
     \end{align*}
   \end{proof}

   \begin{theorem}\label{thm:vtfone} 
     Let $\overline{\Delta}\rightarrow\overline{\Gamma}$ be a
     finite regular covering of graphs, and let $\Gamma$ be a  
     simplicial graph obtained by subdividing each edge of
     $\overline{\Gamma}$ into at least $r$ parts, with $\Delta$
     the corresponding covering of $\Gamma$.  For any periodic
     set $S\subseteq \zz$ and any $r\geq 4$, the group $G_\Gamma^\Delta(S)$
     is virtually torsion-free.
   \end{theorem}

   \begin{proof}
     If $S=\emptyset$ then $G_\Gamma^\Delta(S)$ is the semidirect
     product $BB_\Delta\semi \pi(\Delta,\Gamma)$ which is clearly
     virtually torsion-free.  In all other cases, we may assume up
     to isomorphism that $0\in S$, and we do so for the rest of this
     proof.
     
     Embed the group $\pi(\overline{\Delta},\overline{\Gamma})=
     \pi(\Delta,\Gamma)$ into a finite alternating group $A_N$.
     Choose a maximal tree $T\subseteq \overline{\Gamma}$, and
     fix an orientation on the edges of $\overline{\Gamma}-T$, so that the 
     fundamental group of $\overline{\Gamma}$ is naturally isomorphic to the
     free group on the set of edges of ${\overline\Gamma}-T$.  The covering
     thus gives rise to a labelling $\sigma$ of the directed
     edges of $\overline{\Gamma}$ by elements of $A_N$, with the properties
     that every edge of $T$ is labelled by the identity element
     and that the product of the labels on the two different
     orientations of the same edge is the identity.  The labelling
     $\sigma$ associates an element of $A_N$ to each directed edge path
     in $\overline{\Gamma}$: if the directed edge path is
     $e_1,e_2,\ldots,e_l$, the associated element is
     $\sigma(e_1)\sigma(e_2)\cdots \sigma(e_l)$.
     By definition, the element of $A_N$
     associated to a closed directed edge path will be the identity
     if and only if this path lifts to a closed path in
     $\overline{\Delta}$.

     The relators in the presentation for $G_\Gamma^\Delta(S)$ given
     in Section~\ref{sec:back} are of the form $e_1^je_2^j\cdots
     e_l^j$, where $e_1,\ldots,e_j$ is a closed directed edge path in
     $\Gamma$ and either $j\in S$ or the path lifts to a closed path
     in $\Delta$.  Moreover, every non-identiy element of finite order
     in $G_\Gamma^\Delta(S)$ is conjugate to an element of the form
     $e_1^je_2^j\cdots e_l^j$, where $j\notin S$ and $e_1,\ldots,e_l$
     is a closed directed edge path in $\Gamma$ whose lift to $\Delta$
     is not a closed path.  It follows that to construct a
     homomorphism with torsion-free kernel from $G_\Gamma^\Delta(S)$
     to a finite group, it suffices to construct a labelling $\mu$ of
     the directed edges of $\Gamma$ by the elements of a finite group
     so that for $j\in S$ and for every closed directed edge path
     $e_1,\ldots,e_l$ in $\Gamma$, $\mu(e_1)^j\mu(e_2)^j\cdots
     \mu(e_l)^j=1$, while for $j\notin S$ we have that
     $\mu(e_1)^j\mu(e_2)^j\cdots \mu(e_l)^j=1$ if and only if
     $e_1,\ldots,e_l$ lifts to a closed path in $\Delta$.

     Fix some $n>0$ with $S+n=S$.  First we consider the case
     when $S=\zz-(k+n\zz)$ for some $k$ with $1\leq k<n$.  In
     this case, the finite group that will be the
     target of our labelling $\mu$ is the wreath product $S_N\wr
     C_n$.  To ease the notation, we fix an orientation on each of the
     edges of $\overline{\Gamma}$ and define the labelling $\mu$
     on those directed edges of $\Gamma$ that are oriented in
     the same direction as our chosen orientation on the edge
     of $\overline{\Gamma}$ that they are contained in.  
     If $e$ is a directed edge of $\Gamma$ with our chosen
     orientation, $\sigma(e)$ is an element of the alternating
     group $A_N$.  It is known that every element of $A_N$ is
     equal to the commutator of a pair of elements of
     $S_N$~\cite{ore}.  Hence we may choose 
     $\alpha(e),\beta(e)\in S_N$ with $\sigma(e)=\alpha(e)\beta(e)
     \alpha(e)^{-1}\beta(e)^{-1}$.  Define elements
     $a(e),b(e),c(e),d(e)$ of $S_N\wr C_n$ as in the statement of 
     Proposition~\ref{prop:hom}.  If the directed path in $\Gamma$
     that maps homeomorphically to $e$ with its given orientation
     is $e_1,\ldots,e_r$ where $r\geq 4$, define the labelling
     $\mu$ on these edges by
     \[\mu(e_i)=\begin{cases} 
     a(e)&i=1\\
     b(e)&i=2\\
     c(e)&i=3\\
     d(e)&i=4\\
     1&i>4.
     \end{cases}\]
     By Proposition~\ref{prop:hom}, for this labelling we have
     that
     \[\mu(e_1)^j\mu(e_2)^j\mu(e_3)^j\cdots\mu(e_r)^j=\begin{cases}
     \sigma(e)_k&j\equiv k\,\,\hbox{modulo}\,\,n\\
     1&j\not\equiv k \,\,\hbox{modulo}\,\,n.
     \end{cases}
     \]
     Here $\sigma(e)_k$ denotes the copy of $\sigma(e)$ inside the
     $k$th direct factor in $(S_N)^n<S_N\wr C_n$.  This completes
     the proof in the case when $S=\zz-(k+n\zz)$.

     For the general case, rename the labelling $\mu$ used above
     as $\mu^{(k)}$, to emphasize the dependence on $k$.  If $S$
     is any set with $0\in S$ and $S+n=S$, define a finite set
     $\{k_1,\ldots,k_l\}$ as $\{1,\ldots,n-1\}-S$.  For this $S$,
     define a new labelling $\bmu$ of the edges of $\Gamma$ by 
     elements of $(S_N\wr C_n)^l$,
     where the label attached to the edge $e$ of $\Gamma$ is
     \[\bmu(e)=(\mu^{(k_1)}(e),\mu^{(k_2)}(e),\ldots,\mu^{(k_l)}(e)).\]
     The labelling $\bmu$ has the property that the product of
     the $j$th powers of the labels around any closed path in
     $\Gamma$ is the identity if $j\in S$, whereas for $j\in S$
     the product of the $j$th powers of the labels around a
     closed path is equal to the identity if and only if the
     path lifts to a closed path in $\Delta$.  Hence the kernel
     of the corresponding homomorphism is torsion-free as required.
\end{proof}

\begin{proof} (Theorem~\ref{thm:vtorfree})
  By Theorem~\ref{thm:vtfone} we see that the group $G_\Gamma^\Delta(S)$ is
  virtually torsion-free in the case when $S$ is periodic and $\Gamma$
  is a graph obtained from another graph $\overline{\Gamma}$ by
  subdividing each edge into at least four pieces.  If $M\rightarrow
  L$ is a covering obtained by pulling back the regular covering
  $\Delta\rightarrow \Gamma$ along a simplicial map $f:L\rightarrow
  \Gamma$, then $\pi(M,L)$ is a subgroup of $\pi(\Delta,\Gamma)$,
  and every finite subgroup of $G_L^M(S)$ maps isomorpically to a
  subgroup of $G_\Gamma^\Delta(S)$ under the map induced by $f$.
  Hence in this case $G_L^M(S)$ is also virtually torsion-free.  

  Now suppose that $\widehat{\Gamma}$ is obtained from
  $\overline{\Gamma}$ by subdividing each edge into exactly 4~pieces,
  and that $\Gamma$ is obtained from $\overline{\Gamma}$ by
  subdividing each edge into at least 12 pieces.  In this case, there
  is a map from $\Gamma$ to $\widehat{\Gamma}$ with the property that
  the image of any three consecutive edges of $\Gamma$ is either a
  vertex or a single edge of $\widehat{\Gamma}$.  In more detail,
  suppose that $\overline{e}$ is a directed edge of
  $\overline{\Gamma}$ that is subdivided into
  $\widehat{e}_1,\widehat{e}_2, \widehat{e}_3,\widehat{e}_4\in
  \widehat{\Gamma}$ and into $e_1,\ldots,e_r\in\Gamma$ with $r\geq
  12$.  In this case, such a map is given explicitly by mapping
  $e_2,e_5,e_{r-4},e_{r-1}$ homeomorphically to the edges
  $\widehat{e}_1,\widehat{e}_2, \widehat{e}_3,\widehat{e}_4$
  respectively and collapsing each other $e_j$ to a point.
  Thus if $M\rightarrow L$ is obtained by pulling back the
  covering $\Delta\rightarrow \Gamma$, then both 
  the hypotheses of Theorem~\ref{thm:bothvspec} and
  Proposition~\ref{prop:vspec} are satisfied and $G_L^M(S)\semi n\zz$
  is virtually special, which implies that Conjectures
  \ref{conj:vtorfree}~and~\ref{conj:resfin} hold in this case.
  \end{proof} 

\begin{remark}
  In Theorem~\ref{thm:vtorfree} and Theorem~\ref{thm:vtfone},
  the hypothesis that each
  edge of $\overline{\Gamma}$ be subdivided into at least $r$
  pieces can be replaced by a slightly weaker hypothesis:
  it is sufficient for each edge of $\overline{\Gamma}-T$ to be
  subdivided into at least $r$ pieces.  
\end{remark} 

\begin{proof} (Corollary~\ref{cor:vtorfree})
  Let $\overline{\Gamma}$ be a rose (i.e., a 1-dimensional
  CW-complex with one vertex) whose fundamental group is
  isomorphic to the free group $F$, and fix such an isomorphism.
  A standard argument of obstruction theory shows that the 
  homomorphism $f:\pi_1(L)\rightarrow F$ is induced by some 
  continuous map
  $\phi:|L|\rightarrow \overline{\Gamma}$~\cite[Ch.~4.3]{hatcher}.  To
  see this, view $|L|$ as a CW-complex, with cells the topological
  realizations of the simplices of $L$.  Pick a maximal tree $T$ in $L$,
  and send every 0-cell and every 1-cell in $|T|$ to the 0-cell of
  the rose.  Every other 1-cell $|\sigma|$ of $|L|$ represents a unique word
  in the free generators of $F$, and this word can be used to define
  $\phi|_{|\sigma|}$.  Assume by induction that the map $\phi$ has been
  defined on the $n-1$-skeleton of $L$ for some $n\geq 2$.  Since the
  universal cover of the rose $\overline{\Gamma}$ is a tree, the higher
  homotopy groups of $\overline{\Gamma}$ are all trivial.  Thus for each
  $n$-simplex $\sigma$, the map from the $(n-1)$-sphere to $|\overline{\Gamma}|$
  defined as the restriction of $\phi$ to the boundary of $\sigma$ can be
  extended to a map from the $n$-disc to $|\overline{\Gamma}|$.  By doing
  this for each $n$-simplex, one extends $\sigma$ to the $n$-skeleton of
  $|L|$.

  Let $\Gamma$ be obtained from the rose $\overline{\Gamma}$ by
  subdividing each edge into 12~pieces.  
  By the simplicial approximation theorem~\cite[Ch.~2.C]{hatcher},
  there is an iterated barycentric subdivision $L'$ of $L$ with
  respect to which the map $\phi:|L'|=|L|\rightarrow |\Gamma|=
  |\overline{\Gamma}|$ is homotopic to a simplicial map
  $\psi:L'\rightarrow \Gamma$.  If $\Delta$ is the regular covering
  of $\Gamma$ corresponding to the finite-index normal subgroup
  $N\triangleleft F$, the induced cover of $L'$ is a (possibly
  not connected) regular covering of $L'$ with $F/N$ as its
  group of deck transformations.  The fundamental group of
  each component of this covering is $f^{-1}(N)$, and we may
  take $M'$ to be one of these components.
\end{proof}

\section{Some torsion-free-by-cyclic examples}\label{sec:eight}

\begin{proposition}\label{prop:homology} 
  Let $p$ be a prime and suppose that $S=p\zz$ and that $M\rightarrow L$
  is a connected $p$-fold regular covering.  Then $G_L^M(S)\semi p\zz$
  is virtually special and so all of our conjectures hold in this case.
\end{proposition}

\begin{proof}
  The $p$-fold covering $M$ is classified by an element of
  $H^1(L;\zz/p\zz)$, so let $f:L^1\rightarrow \zz/p\zz=\ff_p$ be a
  cocycle representing this cohomology class.  The cocycle
  $f$ extends to a group homomorphism $G_L^M(S)\rightarrow \ff_p$
  and this homomorphism is easily seen to have torsion-free kernel.
  The claim now follows from Corollary~\ref{cor:abelian}.
\end{proof}

\section{An example in detail} \label{sec:nine}

We consider now the simplest case of Proposition~\ref{prop:homology}
in detail; the case
when $p=2$, $S=2\zz$ and $L$ is the boundary of a square.  Let the edges
of $L$ be labelled $a,b,c,d$ so that a group presentation for
$G_L^M(S)$ is
\[\langle a,b,c,d \,\,:\,\, a^{2n}b^{2n}c^{2n}d^{2n} = 1= (a^nb^nc^nd^n)^2,
\,\, n\in \zz\rangle.\]
There are fifteen index two subgroups of $G_L^M(S)$, indexed by the
subset of the generators $\{a,b,c,d\}$ consisting of elements not in
the subgroup.  The torsion-free subgroups are those in which
$abcd$ is not contained in the subgroup, or equivalently the
set of generators not in the subgroup has odd cardinality.
The order four rotation of $L$ induces a group of four automorphisms
of $L$, so up to isomorphism there are only two cases to consider:
the subgroup containing $b,c,d$ but not $a$ and the subgroup
containing $d$ but not containing $a,b,c$.

The space $X_L/BB_L$, which is a classifying space for $BB_L=G_L^M(\zz)$,
consists of a union of four 2-dimensional cylinders.  Label the four
vertices of $L$ by $w,x,y,z$, so that the directed edges are
$a=(w,x)$, $b=(x,y)$, $c=(y,z)$ and $d=(z,w)$.  Let $P$ denote a copy
of the plane $\rr^2$, tesselated by squares with vertex set $\zz^2$
and 1-skeleton $(\zz\times \rr)\cup (\rr\times \zz)$.  Each of the four
cylinders making up $X_L/BB_L$ is isomorphic to the quotient of $P$
by the subgroup generated by $(-1,1)$, with the height function
on $P$ and on $P/\langle (-1,1)\rangle$ given by $(s,t)\mapsto s+t$.
In $P/\langle (-1,1)\rangle$, the images of the horizontal edges
all belong to one hyperplane and the images of the vertical edges
all belong to a second hyperplane.

If $H$ is any of the eight torsion-free index two subgroups of
$G_L^M(2\zz)$, then $X_L^M(2\zz)/H$ is a 2-fold branched covering of
$X_L/BB_L$, with branching only at the vertices of even height.  To
better understand $X_L^M(2\zz)/H$, we first describe the subcomplexes
$X_a$, $X_b$, $X_c$ and $X_d$ consisting of the inverse images of the
four cylinders of $X_L/BB_L$.  The isomorphism type of such a
subcomplex depends only on whether the letter that labels it is
contained in the subgroup $H$ or not, so consider $X_h$ for some
$h\in \{a,b,c,d\}$.  The link of 
an unbranched vertex of $X_L^M(2\zz)/H$ is a copy of the
octahedralization of $L$ and the link at a branched vertex is a copy
of the octahedralization of $M$.  Since the inverse image in $M$ of
each edge of $L$ is a disjoint union of two edges, the link of a
branch vertex inside $X_h$ is a disjoint 
union of two squares (i.e., the octahedralization of a pair of
disjoint edges).  If $h\in H$, then $X_h$ 
consists of two copies of $P/\langle (-1,1)\rangle$, with each vertex
of odd height in one copy identified with the vertex in the other copy
of the same height.  If $h\notin H$, then instead $X_h$ consists 
of one copy of a larger cylinder $P/\langle (-2,2)\rangle$, in
which the two vertices of each odd height are identified with each
other.

To study the hyperplanes in the whole complex, we first
consider the hyperplanes in $X_h$.  Suppose that $u,v\in \{w,x,y,z\}$
are the vertices of the edge $h$.  If $h\in H$, then $X_h$ contains
two hyperplanes labelled $u$ and two hyperplanes labelled $v$, one
of each type in each of the two copies of $P/\langle (-1,1)\rangle$.
Each $u$-hyperplane intersects exactly one of the two
$v$-hyperplanes.  In this case, viewing it as a complex in its own
right, $X_h$ is special.  If on the other hand $h\notin H$, then as before
there are two $u$-hyperplanes and two $v$-hyperplanes, but this time
each $u$-hyperplane intersects each $v$-hyperplane.  Furthermore, at
every branch vertex two $u$-edges and two $v$-edges of each height
all meet.  The two edges with the same label at the same height belong
to different hyperplanes.  Hence in the case when $h\notin H$, each
$u$-hyperplane interosculates with each $v$-hyperplane, and $X_h$
itself is not special.  Note also that if we define a line to be
the image in $X_h$ of either $\rr\times \{n\}$ or $\{n\}\times \rr$
for some integer $n$, then the edges in a single line alternate
between the two hyperplanes of $X_h$ labelled by the relevant
letter.  

It follows from the above considerations that the complex
$X_L^M(2\zz)$ is never special.  In the case when $a\notin H$
and $b,c,d\in H$, the two $w$-hyperplanes in $X_a$ become
identified in $X_d$, because the intersection of $X_a$ and
either of the cylinders of $X_d$ consists of a line that
contains edges from both $w$-hyperplanes of $X_a$.  Similarly, 
the two $x$-hyperplanes in $X_a$
become identified in $X_b$.  Hence the whole complex
contains one $w$-hyperplane and one $x$-hyperplane, each
of which self-osculates.  The $w$-hyperplane and the
$x$-hyperplane also interosculate.  There are two $y$-hyperplanes
and two $z$-hyperplanes which are not involved in any
self-osculation or inter-osculation.

In the case when $a,b,c\notin H$ and $d\in H$, the two
$z$-hyperplanes in $X_c$ become identified in $X_d$ and
the two $w$-hyperplanes in $Z_a$ become identified in $X_d$.
Thus there is just one $z$-hyperplane and one $w$-hyperplane,
each of which self-osculates.  There are two $x$-hyperplanes
and two $y$-hyperplanes, each of which does not self-osculate.
However, any pair of hyperplanes labelled by the distinct
ends of an edge interosculate with each other.

Thus we see that for $H$ any of the torsion-free index two
subgroups of $G_L^M(2\zz)$, the complex $X_L^M(2\zz)/H$ fails
to be special.  The proof of Corollary~\ref{cor:abelian}
tells us that there is an index~16 normal subgroup $H\leq G_L^M(2\zz)$
such that $X_L^M(2\zz)/H$ is special and since the quotient group has
exponent~2, it follows that this $H$ is the kernel of 
the map to $H_1(G_L^M(2\zz);\ff_2)$.

It can also be seen directly that this covering is special.
In $X_L^M(2\zz)$, ignoring
for now the identification of vertices that is responsible for
the branching, the inverse image of each of the four cylinders
of $X_L/BB_L$ consists of 8 copies of the cylinder
$P/\langle (-2,2)\rangle$.  The edges of a given height labelled
by each fixed letter form a single free orbit for the action of
$Q=H_1(G_L^M(2\zz);\ff_2)\cong (C_2)^4$.  It can be shown that these
edges all lie in distinct hyperplanes, so that there 
are 16 distinct hyperplanes labelled with each letter.  The
vertices of odd height form a single $Q$-orbit of type
$Q/\langle abcd\rangle$, where we have identified the element
$abcd$ of $G_L^M(2\zz)$ and its image in $Q$.  This already implies
that no self-osculation or interosculation can occur, without
considering cylinder equivalence.  However, to illustrate the special
case of our general argument, we discuss cylinder equivalence.
The cylinder-equivalence classes of edges labelled $x$ correspond
to the cosets $Q/\langle a,b\rangle$ and the cylinder-equivalence
classes of edges labelled $y$ correspond to the cosets $Q/\langle
b,c\rangle$.  Since $abcd\notin \langle a,b,c\rangle$, one sees
that if $e,e'$ are incident edges labelled $x$~and~$y$, then no
edge cylinder equivalent to $e$ can be incident on any edge
cylinder equivalent to either $e$ or $e'$, except for $e,e'$
themselves.  The cylinder equivalence classes for other edges
are similar.

\section{Two applications}\label{sec:ten}

In this section we use the cases of our conjectures that we have
established to construct some groups with surprising combinations
of properties.  

\begin{theorem}
  For each $m\geq 6$ there is a finitely generated group $G_m$ with an
  infinite presentation satisfying the $C'(1/m)$ small cancellation
  condition with the properties that $G_m$ is residually finite,
  torsion-free and embeds
  in a finitely presented group, but the word problem for $G_m$ is insoluble.  
\end{theorem}

\begin{proof}
  Fix some integer $l\geq 2m+1$, 
  let $L$ be a circle triangulated as the boundary of a $l$-gon,
  and let $M$ be the universal cover of $L$.  The group $G_m$ will
  be the group $G_L^M(T)$ for a suitable set $T\subseteq\zz$.  Any such group
  is torsion-free by Proposition~\ref{prop:torfree}.
  As discussed above, this group has the presentation
  \[G_m=\langle\, a_1,\ldots, a_l\,\,:\,\, a_1^na_2^n\cdots a_l^n
  \,\,n\in T\,\rangle.\] The choice of $l\geq 2m+1$ implies that for
  each $T$ this presentation satisfies the $C'(1/m)$ condition.  The
  boundary of the $l$-gon may be viewed as a subdivision of the 1-edge
  CW-structure on the circle, and so since $l\geq 12$ the hypotheses
  of Theorem~\ref{thm:vtorfree} are satisfied.  The subset $T$ that we
  will choose will be of the form $T=T(S)$ as in the statement of
  Proposition~\ref{prop:tsclosed}, for some $S\subseteq \nn$.  Each
  such set is closed in the profinite topology on $\zz$ so by
  Theorem~\ref{thm:vtorfree} $G_m$ is residually finite.
  By~\cite[lemma~15.3]{ufp} and the related discussion in
  Section~\ref{sec:three}, the element $a_1^n\cdots a_l^n$ is equal to
  the identity in $G_m$ if and only if $n\in T=T(S)$.  By
  Proposition~\ref{prop:tsclosed}, if $S$ is recursively enumerable
  but $\nn-S$ is not (so that $S$ is not recursive), there can be no
  algorithm to decide membership of $T(S)$ and so the word problem for
  $G_m$ is insoluble.  Since $T(S)$ is recursively enumerable the
  Higman embedding theorem~\cite[ch.~IV.7]{lynsch} tells us that $G_m$
  can be embedded in a finitely presented group.
\end{proof} 
  
Examples of residually finite groups with insoluble word problem that
can be embedded in finitely presented groups were constructed in the
1970's by Dyson and by Meskin~\cite{dyson,meskin}, but their examples
contain torsion.  

For any $L$ and for $M=\widetilde L$, it has been shown that
$G_L^M(S)$, $S\neq \zz$, has soluble word problem if and only if
$\pi_1(L)$ has soluble word problem and $S$ is
recursive~\cite[thm.~6.4]{bks}.  A direct proof of this can be given
in the case when $L$ is the boundary of an $l$-gon for $l\geq 13$ as
in the theorem above.  Since $a_1^n\cdots a_l^n=1$ if and only if
$n\in S$, a solution to the word problem implies that $S$ is
recursive.  Conversely, given a word of length $N$ in the $a_i$, if
$S$ is recursive we may list the elements of $S\cap [-N/l,N/l]$ and
thus list the relators in the given presentation of length at most
$N$.  Since this presentation satisfies the $C'(1/6)$ condition, any
word of length $N$ that is equal to the identity will contain more
than half of a dihedral permutation of a one of these relators as a
subword.

\begin{proposition}
  Let $l\geq 12$ and let $G$ be given by the presentation 
  \[G=\langle a_1,a_2,\ldots,a_l\,\,:\,\,(a_1^na_2^n\cdots a_l^n)^2=1,\,\, n\in
  \zz\rangle.\]
  Then $G$ is residually finite, but $G$ is not virtually torsion-free and not
  linear in characteristic zero.  Every finite subgroup of $G$ has
  order at most~2.
\end{proposition}

\begin{proof}
  Let $M\rightarrow L$ be the 2-fold cover of the $l$-gon.  The group
  $G$ given above is isomorphic to $G_L^M(\{0\})$.  Any finite subset
  of $\zz$ is closed in the profinite topology, and any non-empty
  finite subset is not periodic.  Hence this group is residually
  finite by Theorem~\ref{thm:vtorfree}, and is not virtually
  torsion-free by Theorem~\ref{thm:properties}.  Every non-trivial
  finite subgroup of $G$ is conjugate to the group generated by
  $a_1^n\cdots a_l^n$ for some $n\neq 0$ and has order two.  Any
  finitely generated linear group in characteristic zero is virtually
  torsion-free~\cite{alperin}, and so $G$ cannot be linear.
\end{proof}

\leftline{\bf Author's addresses:}

\obeylines

\smallskip
{\tt i.j.leary@soton.ac.uk} \qquad {\tt v.vankov@soton.ac.uk} 

\smallskip
School of Mathematical Sciences, 
University of Southampton, 
Southampton,
SO17 1BJ

\end{document}